\documentclass{amsart} 
\usepackage{graphicx,amsmath,amscd,amssymb,tensor,xfrac}
\usepackage[nohug]{diagrams}

\newtheorem*{thm:main}{Theorem \ref{brane local form}}
\newtheorem*{thm:mainlem}{Main Lemma \ref{main lemma}}
\usepackage{tikz,units}
\usepackage{mathtools}
\usepackage{hyperref}

\newtheorem{thm}{Theorem}[section]
\newtheorem{cor}[thm]{Corollary}
\newtheorem{lem}[thm]{Lemma}
\newtheorem{prop}[thm]{Proposition}
\newtheorem*{thm A}{Theorem A}
\newtheorem*{thm B}{Theorem B}
\newtheorem*{lem:appendix}{Lemma \ref{appendix lemma}}
\theoremstyle{definition}
\newtheorem{defn}[thm]{Definition}

\newtheorem{example}[thm]{Example}

\theoremstyle{remark}
\newtheorem{rem}[thm]{Remark}
\numberwithin{equation}{section}

\newcommand{\longversion}[1]{}

\newcommand{\bb}[1]{\mathbb{#1}}

\newcommand{\II}{\mathbb I}
\newcommand{\JJ}{\mathbb J}
\newcommand{\R}{\mathbb R}
\newcommand{\C}{\mathbb C} 

\newcommand{\Z}{\mathbb Z}

\newcommand{\Tc}{\mathbb T} 
\renewcommand{\L}{\mathcal{L}}

\newcommand{\T}{\textsf{T}} 

\newcommand{\tens}{\otimes} 
\newcommand{\dsum}{\oplus}

\newcommand{\iso}{\cong} 
\newcommand{\isoto}{\overset\sim\to}

\renewcommand{\phi}{\varphi} 
\renewcommand{\to}{\longrightarrow}

\newcommand{\oto}[1]{\overset{#1}\to}

\newcommand{\into}{\hookrightarrow}
 
\renewcommand{\mapsto}{\longmapsto}

\newcommand{\comp}{\circ}

\newcommand{\Ann}{\textrm{Ann}} 
 
\newcommand{\del}{\partial}

\newcommand{\pair}[1]{\left\langle #1 \right\rangle}

\renewcommand{\^}{\wedge}

\renewcommand{\epsilon}{\varepsilon}

\newcommand{\hide}[1]{} 

\newcommand{\Id}{\textrm{Id}}

\renewcommand{\Im}{\textrm{Im}}

\newdiagramgrid{house}%
{}{.5,.5,1,1,1,1,1}

\begin{document}

\title{The local structure of generalized complex branes}

\author{Michael Bailey}

\begin{abstract}
We show (modulo a parity condition) that, a generalized complex brane in a generalized complex manifold is locally equivalent to a holomorphic coisotropic submanifold of a holomorphic Poisson structure, with higher-rank branes corresponding to holomorphic Poisson modules. We describe (but do not prove here) the global version of this holomorphicity result. Finally, we use the ``local holomorphic gauges'' to give examples, in the Hopf surface with nonstandard generalized complex structure, of branes which are neither Lagrangian nor complex. 
\end{abstract}

\maketitle

\tableofcontents

\section{Introduction}

Kapustin and Orlov \cite{KapustinOrlov2003} proposed, for string theory reasons, that the ``branes'' of the A-model (or generalized Fukaya category) of a symplectic manifold should include, not only Lagrangians with flat vector bundles, but also ``coistropic A-branes,'' i.e., coisotropic submanifolds equipped with vector bundles and certain additional data.  Later, Gualtieri \cite{Gualtieri2011} showed that these branes have their natural definition as subobjects in \emph{generalized complex geometry}, of which symplectic geometry is a special case.

In \cite{bailey2013}, we showed that a generalized complex structure is locally equivalent to a product of a holomorphic Poisson structure with a symplectic structure (and, modulo a parity condition, we can say that it is locally equivalent just to a holomorphic Poisson structure). This yields natural examples of branes via \emph{holomorphic coisotropic submanifolds} and their holomorphic modules. In this paper, we prove the local converse, namely, generalized complex branes are  (again, modulo a parity condition) \emph{locally equivalent} to holomorphic Poisson modules over holomorphic coisotropics.

The ``rank-0'' case, which deals only with the submanifold and its ``generalized tangent bundle,'' is covered in Theorem \ref{brane local form}. The higher-rank case, which considers also the data of a vector bundle over the submanifold, equipped with a certain connection, is covered in Corollary \ref{higher rank local form}.

The holomorphic Poisson structure to which a given generalized complex structure is equivalent is not unique in general---there is a class of gauge transforms known as \emph{$B$-field transforms} which can relate holomorphic Poisson structures with different underlying complex structures. Given a generalized complex structure and such a choice of ``holomorphic gauge'', we do not expect a particular given brane to be a complex submanifold. The choice of gauge must be adapted to the brane.

For example, if $(M,I,\Omega)$ is a holomorphic symplectic manifold, its generalized complex branes correspond to real-Lagrangian submanifolds of the corresponding real symplectic manifold $(M,\omega=\Im(\Omega))$. Most of these Lagrangians will not be $I$-complex submanifolds. However, given any real-Lagrangian $S \subset M$ and a point $p \in S$, there exists a complex structure in a neighbourhood of $p$ such that $S$ is a complex submanifold and $\omega$ is holomorphic (this follows easily from the Darboux normal form).

\subsection{Morita equivalence, stacks, and the global picture}\label{upcoming paper section}

This paper is strictly about the local holomorphic structure of generalized complex branes. In an upcoming paper, we describe how branes may be seen as globally holomorphic in a weak sense.

In \cite{BG2016}, Bailey and Gualtieri show that---given an integrability condition---a generalized complex structure is equivalent to a real Poisson structure with a compatible complex structure defined only on the associated \emph{stack}. In other words, the generalized complex structure is determined by the underlying real Poisson manifold and a Morita equivalence to a holomorphic Poisson manifold (which is unique only up to holomorphic Morita equivalence).

There is a global version of the results of this paper: namely, that a generalized complex brane is just a coisotropic which holomorphic \emph{as a sub-object of the associated stack}. To be more concrete: a brane in a generalized complex manifold $(M,\II)$ may be represented by a coisotropic submanifold $S$ in the underlying Poisson manifold $(M,P)$, together with a Morita equivalence to a holomorphic Poisson manifold $(X,\pi)$ (a ``holomorphic atlas'', which determines the generalized complex structure), together with a Lagrangian sub-Morita-equivalence between $S$ and a holomorphic coistropic $\Sigma \subset X$.

\section{Generalized complex geometry and branes}

We recall some of the basics of generalized complex geometry.  The canonical reference is \cite{Gualtieri2011}.

Generalized complex structures are defined on \emph{exact Courant algebroids}. We skip the axiomatic definition of an exact Courant algebroid, since it is in any case isomorphic to one of the following concrete models:
\begin{defn}
Let $M$ be a manifold, and let $H$ be a closed 3-form on $M$. The \emph{standard Courant algebroid} on $M$ with \emph{twist $H$} is the vector bundle $\Tc M := TM \dsum T^*M$ equipped with a bracket $[\cdot,\cdot] : C^\infty(\Tc M \dsum \Tc M) \to C^\infty(\Tc M)$ defined as follows: for $X,Y \in C^\infty(T M)$ and $\xi,\eta \in C^\infty(\T^*M)$,
\begin{align}
[X+\xi,Y+\eta] = [X,Y]_{Lie} + \L_X \eta - \iota_Y \xi + \iota_x \iota_Y H.
\end{align}
Also part of the structure of a Courant algebroid is the \emph{anchor map} to $TM$ (in this case, the canonical projection $TM \dsum T^*M \to TM$) and a nondegenerate, symmetric pairing (in this case, the canonical symmetric $\pair{\cdot,\cdot} : TM \dsum TM \to \R$).
\end{defn}

\begin{defn}
A \emph{generalized complex structure} on $\Tc M$ is a complex structure $\II : \Tc M \to \Tc M, \II^2 = -1$ which is orthogonal with respect to the pairing and whose $+i$--eigenbundle is involutive with respect to the bracket.
\end{defn}

The two basic examples are, for a symplectic structure $\omega : TM \to T^*M$ and a complex structure $I : TM \to TM$,
\begin{align}\label{basic gc examples}
\II_\omega =
\begin{pmatrix}
0 & \omega^{-1} \\
-\omega & 0
\end{pmatrix}
\quad \textnormal{and} \quad
\II_I =
\begin{pmatrix}
-I & 0 \\
0 & I^*
\end{pmatrix}.
\end{align}

An isomorphism of Courant algebroids is an isomorphism of the underlying vector bundles which respects the bracket and the pairing (and, consequently, the anchor map).  The Courant isomorphisms, or \emph{generalized diffeomorphisms}, are generated by diffeomorphisms---acting by pushforward and inverse pullback---and the \emph{$B$-transforms}: for $B$ a closed 2-form, the map
\begin{align}
e^B &: \Tc M \to \Tc M \notag \\
&: X + \xi \mapsto X + \iota_X B + \xi, \notag \\
\textnormal{i.e.,} \quad e^B &=
\begin{pmatrix}
1 & 0 \\
B & 1
\end{pmatrix}.
\end{align}
is a generalized diffeomorphism.  If $B$ is not closed, the $B$-transform is a generalized diffeomorphism to a Courant algebroid with twist $H + dB$.

\begin{rem}
A $B$-transform with $dB + H=0$ determines (and is equivalent to) a new choice of isotropic, involutive splitting of the sequence $T^*M \into \Tc M \to TM$, i.e.,
\begin{align}
s : TM \to e^B(TM) \subset \Tc M.
\end{align}
\end{rem}

\begin{example}[Holomorphic Poisson]
We give one more fundamental example of a generalized complex structure.  If $I$ is a complex structure and $P : T^*M \to TM$ is a real Poisson bivector for which $\pi := -\tfrac{1}{4}(IP + iP) = -\tfrac{i}{2}P^{2,0}$ is holomorphic with respect to $I$, then
\begin{align}\label{gc hol poisson}
\II_{I,P} := \begin{pmatrix}
-I & P \\
0 & I^*
\end{pmatrix}
\end{align}
is generalized complex.  Conversely, if a generalized complex structure has the matrix form \eqref{gc hol poisson}, then $I$ is a complex structure and $\pi := -\tfrac{i}{2}P^{2,0}$ is holomorphic Poisson. In this case, $\II_{I,P}$ has $+i$-eigenbundle
\begin{align}
L = T_{0,1} \dsum (1+\pi) (T^*_{1,0}).
\end{align}
\end{example}

\begin{rem}
For \emph{any} generalized complex structure $\II$, the ``upper right'' (i.e., $T^* \to T$) component defines a real Poisson structure \cite{Gualtieri2011}; unlike the other matrix entries, this is invariant under conjugation by a $B$-transform, and so it is indeed an invariant of the generalized complex structure. (On an abstract Courant algebroid, its definition is intrinsic.)

However, if $\II$ does not have the form \eqref{gc hol poisson}, i.e., if the lower-left ($T \to T^*)$ component doesn't vanish, we don't expect the diagonal components to define integrable complex structures or even to square to $-1$.
\end{rem}

In \cite{bailey2013}, we showed:
\begin{thm}\label{gc local} 
Any generalized complex structure is locally equivalent (up to $B$-transform) to a product of a symplectic structure (as in \eqref{basic gc examples}) with a holomorphic Poisson structure (as in \eqref{gc hol poisson}).  Moreover, if the corresponding real Poisson structure has rank $0 \mod 4$, then in fact the generalized complex structure is locally equivalent to just a holomorphic Poisson structure.
\end{thm}

\begin{rem}
A generalized complex structure, $\II$, does not in general determine a complex structure.  The theorem says that, near any point, there is a choice of complex structure $I$ and a $B$-field such that $e^B \II e^{-B}$ has the form \eqref{gc hol poisson}.  In fact, $I$ is not unique and may not exist globally.  Different ``holomorphic gauges'' are related, in the holomorphic Poisson category, by \emph{Morita equivalence} as explained in \cite{BG2016}. The way that branes appear in these Morita equivalences, and thus the global holomorphic picture of branes, is the subject of the upcoming paper described in Section \ref{upcoming paper section}.
\end{rem}

\subsection{Generalized complex branes}
Suppose $\Tc M = TM \dsum T^*M$ is the standard Courant algebroid on $M$ with twist $H$, and suppose that $\iota : S \into M$ is a submanifold.

A \emph{Courant trivialization} along $S$ is a 2-form $F$ on $S$ such that $dF = \iota^*(H)$. Note that $F$ is equivalent to a choice of flat, isotropic splitting, $s : TS \to e^F(TS) \subset \Tc S$, of the ``pullback Courant algebroid'' $\Tc S = TS \dsum T^*S$ with twist $\iota^*(H)$.

A Courant trivialization along $S$ determines a \emph{generalized tangent bundle}
\begin{align}
\tau_S &= s(TS) + N^*S \subset \Tc M|_S \notag \\
&= \{X + \xi \in \Tc M|_S \;|\; \iota^*(\xi) = F(X) \}
\end{align}
\begin{rem}\label{generalized tangent is locally standard}
As a vector bundle, $\tau_S \iso TS \dsum N^*S$. Locally about any point in $S$, there is a choice of splitting of $\Tc M$ in which $\tau_S$ is identified with $TS \dsum N^*S$; namely, extend $F$ to a 2-form $\tilde{F}$ on a neighbourhood in $M$ such that $d\tilde{F} = H$.
\end{rem}

\begin{defn}
If $\II$ is a generalized complex structure on $\Tc M$, then a \emph{rank-$0$ generalized complex brane} supported on $S \subset M$ is a choice of Courant trivialization along $S$---and thus a choice of generalized tangent bundle $\tau_S$---such that $\II \tau_S = \tau_S$.
\end{defn}

$\tau_S \tens \C$ decomposes into $+i$ and $-i$-eigenbundles, which we call $\ell$ and $\bar\ell$ respectively.  $\ell$ (and $\bar\ell$) inherit brackets from $\Tc M$ (via arbitrary extension of sections), making them complex Lie algebroids over $S$.

\begin{defn}
A \emph{generalized complex brane} supported on $S$ consists of the data of a rank-$0$ brane, along with a complex vector bundle $V \to S$ and a flat $\ell$-connection
\begin{align}
\nabla : C^\infty(\ell \dsum V) \to C^\infty(V).
\end{align}
\end{defn}

We assume the reader is familiar with Lie algebroids and their connections, though their general theory does not play a big role in this paper. For a comprehensive survey, see \cite{mackenzie_2005}.

\begin{example}\label{lagrangian example}
	Let $(M,\II)$ be a generalized complex manifold of symplectic type, i.e., whose real Poisson structure is invertible to some symplectic form $\omega$.  The prototypical example is $\II_\omega$ in \eqref{basic gc examples}, but in general $\II$ could be a $B$-transform of $\II_\omega$. Then any Lagrangian submanifold $S \subset M$ is a generalized complex brane. In the example $\II_\omega$, $\tau_S = TS \dsum N^*S$, but in general $\tau_S$ is fixed uniquely by $\tau_S = \II(N^*S) \dsum N^*S$.
	
	In this case, $\ell = \tau_S^{1,0} \iso \C \tens TS$ as Lie algebroids, so higher-rank branes are just flat vector bundles on $S$.
\end{example}

\begin{rem}
Generalized complex branes are, in particular, coisotropic submanifolds. Namely, if $M$ is a generalized complex manifold with corresponding real Poisson structure $P:T^*M \to TM$, and if $S \subset M$ is the support of a generalized complex brane, then $P(N^*S) \subset TS$.

The distribution $P(N^*S)$ induces a foliation, $F_S$, on $S$. ($TF_S = TS$ in the Lagrangian example.) Kapustin and Orlov first observed in the physics context that, transverse to $F_S$ in $S$, there should be a holomorphic structure. Indeed, $P(N^*S) \dsum N^*S$ is an $\II$--invariant subbundle of $\tau_S$, so $\II$ determines a complex structure on
\begin{align*}
NF_S \iso \frac{\tau_S}{P(N^*S) \dsum N^*S}.
\end{align*}
\end{rem}

\begin{example}\label{hol coiso example}
Let $(M,I,\pi)$ be a holomorphic Poisson manifold, and $\II_{I,P}$ the corresponding generalized complex structure, as in \eqref{gc hol poisson}. If $S \subset M$ is a holomorphic, coisotropic submanifold, then $S$ supports a generalized complex brane with $\tau_S = TS \dsum N^*S$. In this case,
\begin{align}\label{coisotropic ell}
\ell = \tau_S^{1,0} = T_{0,1}S \dsum (1+\pi) (N^*_{1,0} S).
\end{align}
Given a vector bundle $V \to S$, the first summand of \eqref{coisotropic ell} acts on $V$, making $V$ a holomorphic vector bundle, and the second summand makes it a \emph{holomorphic Poisson module} over $S$ (see Corollary \ref{higher rank local form} or \cite{Gualtieri2010} for the argument, and see \cite{fernandes2000} for an overview of Poisson connections).

Unlike the Lagrangian example, in this case the generalized tangent bundle of $S$ is not uniquely determined by $\II$ and $S$. For example, given the trivial Poisson structure $\pi=0$, $S$ is just a complex submanifold, and $\tau_S$ may be transformed by any closed $(2,0) + (0,2)$--form for which $S$ is isotropic.
\end{example}

\subsection{Generalized flow}

A \emph{generalized vector field}, i.e., a section of $\Tc M$, generates a time-dependent family of generalized diffeomorphisms, in analogy with vector fields generating families of diffeomorphisms.  If $X + \xi$ is the generalized vector field, with $X \in C^\infty(TM)$ and $\xi \in \Omega^1(M)$, then the diffeomorphism part, $\phi_t$, of the family is just the exponential of $tX$.  The $B$-transform part is
\begin{align}\label{flow B field}
B_t = \int_{s=0}^t \phi_s^*(d\xi) ds.
\end{align}

\section{Statement of the theorem}\label{local structure section}

Most of this paper is devoted to proving the following local structure theorem for rank-0 branes:

\begin{thm}\label{brane local form}
	If $S$ is an $\R$-even-dimensional, rank-0 generalized complex brane in $(M,\II)$ with generalized tangent bundle $\tau \subset \Tc M|_S$, then any $p \in S$ has a neighbourhood in $M$ on which the data $\II$, $S$ and $\tau$ are equivalent, under a $B$-transform, to a holomorphic Poisson structure $\pi$ for some complex structure $I$, with $S$ now a complex, coisotropic submanifold with generalized tangent bundle $TS \dsum N^*S$.
\end{thm}

The higher-rank case follows easily from this, and is covered in Corollary \ref{higher rank local form}.

In a neighbourhood of a point, there are potentially many complex structures with respect to which the generalized complex structure is holomorphic Poisson, and we should not expect most of them to be compatible with the brane.  We must make careful choices when producing the local holomorphic gauge if we want this compatibility.  This is realized as a slight modification---or, rather, a specialization---of the original iteration used in the proof of Theorem \ref{gc local}.

In the original iteration, we first specified a local complex structure about the point $p$ (i.e., identifying with a neighbourhood in $\C^n$), such that it agree with the complex structure that $\II$ determined at $p$.  Then, by measuring $\II$'s failure to be holomorphic Poisson with respect to this complex structure, we construct a generalized vector field whose flow takes $\II$ to something which is approximately holomorphic on $\C^n$.  By iterating this process with the usual Nash-Moser convergence tricks, we arrive, in the limit, at a local generalized diffeomorphism taking $\II$ to a holomorphic Poisson structure on a neighbourhood in $\C^n$.

In our new iteration, we start with a stronger ansatz than just adopting the complex structure of $\C^n$.
\begin{defn}\label{ansatz def}
In our \emph{ansatz}, we suppose that we have a generalized complex structure $\II$ defined on a neighbourhood of $p = 0 \in \C^n$ for which
\begin{align*}
\II_p (T_p \C^n) = T_p \C^n \subset \Tc_p \C^n,
\end{align*}
i.e., for which $\II_p$ takes the form \eqref{gc hol poisson}, and for which $\II|_{T_p \C^n}$ agrees with the complex structure.  Furthermore, we suppose that we have a generalized complex brane supported on a linear subspace $S = \C^k \subset \C^n$ ($S$ corresponds to the first $k$ coordinates), with generalized tangent bundle $\tau = TS \dsum N^*S$.
\end{defn}

Setting aside for the moment the question of whether any particular brane may be put in the form of this ansatz, the new iteration is as follows: we construct a generalized vector field whose flow takes $\II$ to something approximately holomorphic, and which is tangent to $S$ in a generalized sense.  Thus (modulo smoothing issues), the ansatz that $S = \C^k$ will not be violated by the flow, and if this iterated flow converges, in the limit we will have a holomorphic Poisson structure and a complex submanifold $S$ (which will happen to be coisotropic).

The steps in establishing Theorem \ref{brane local form} are as follows:
\begin{enumerate}
	\item Specify an ansatz for a generalized complex structure, and a brane thereof (Definition \ref{ansatz def}).
	\item Specify a homotopy operator on the deformation complex which ``preserves the ansatz'' (Section \ref{homotopy section}).
	\item \label{main lemma step} Use the specialized homotopy operator in the above iteration, to conclude that deformations of the complex structure and their branes---if they are small enough, and if they satisfy the ansatz---are locally equivalent to holomorphic Poisson structures and holomorphic coisotropics (Section \ref{the algorithm}).
	\item \label{conditions step} Argue that, given a parity condition, any point in a generalized complex brane has a neighbourhood (in the ambient manifold) satisfying two conditions:
	\begin{enumerate}
	\item It may be put into the form of the ansatz. This is studied first as a linear problem (Section \ref{linear case}), and then as a local problem (in Part 1 of the proof in Section \ref{final proof section}).
	\item On the neighbourhood, the generalized complex structure is a suitably small deformation of the complex structure of the ansatz (Part 2 of the proof in Section \ref{final proof section}).
	\end{enumerate}
\end{enumerate}
Steps \eqref{main lemma step} and \eqref{conditions step} combine to prove the theorem.

\section{The homotopy operator and its generalized flow}\label{homotopy section}

Let us consider the ansatz in more detail.  $\C^n$ has a standard generalized complex structure, with $+i$-eigenbundle $L = T_{0,1} \dsum T^*_{1,0}$.  Let $\iota : S = \C^k \into \C^n$ be the support of our brane, and let $\tau = TS \dsum N^*S \subset \Tc \C^n = T\C^n \dsum T^*\C^n$ be its generalized tangent bundle.  We suppose $\tau_S$ is invariant under some generalized complex structure $\II$ on a polydisc $B_r \subset \C^n$ centred about the origin, and that $\II$ is a deformation of the complex structure, i.e., $\II$ has $+i$-eigenbundle $L_\epsilon = (1+\epsilon)\cdot L$, where
\begin{align}
\epsilon = \underset{\^ ^2 T_{1,0}}{\epsilon_{2,0}} + \underset{T_{1,0} \tens T^*_{0,1}}{\epsilon_{1,1}} + \underset{\^ ^2 T^*_{0,1}}{\epsilon_{0,2}} \in \^ ^2 L^*
\end{align}
(with each component being a section of the bundle indicated underneath). Integrability of $\II$ is equivalent to the Maurer-Cartan equation
\begin{align}\label{MC}
\bar\del \epsilon + \frac{1}{2}[\epsilon,\epsilon].
\end{align}
(This equation of degree-$3$ tensors splits into four equations, in degrees $(3,0)$, $(2,1)$, $(1,2)$ and $(0,3)$.)

\begin{prop}
The $\II$-invariance of $\tau$ is equivalent to three conditions: that $S$ is $\epsilon_{2,0}$--coisotropic, $\epsilon_{1,1}$ acting on $T\C^n$ preserves $TS$, and $S$ is $\epsilon_{0,2}$--isotropic.
\end{prop}

We will construct a generalized vector field whose flow, to first order, eliminates $\epsilon_{1,1}$ and $\epsilon_{0,2}$.  A primary ingredient in this construction is the homotopy operator for $\bar\del$, i.e., an operator
$$P : \Omega^{0,q}\left(\^ ^{p,0} TB_r\right) \to \Omega^{0,q-1}\left(\^ ^{p,0} TB_r\right)$$
such that
\begin{align}\label{homotopy relation}
P\bar\del + \bar\del P = \Id.
\end{align}
(The point is that, for approximately-$\bar\del$-closed tensors, $P$ is approximately a right inverse for $\bar\del$.)  The bracket on $L^*$ extends to $\^ ^\bullet L^*$ in the usual way (\emph{\`{a} la} Schouten).  Given such a $P$ and the extended bracket, we define
\begin{align}\label{generalized homotopy operator}
V(\epsilon) = P([\epsilon_{2,0},P\epsilon_{0,2}] - \epsilon_{1,1} - \epsilon_{0,2}),
\end{align}
which will be our generalized vector field.

Such homotopy operators $P$ are not unique.  They are defined in, eg., \cite[equation 2.2.2]{NijenhuisWoolf}, or more recently---with improved estimates---in \cite{Wang2017}, and elsewhere. The definitions vary between sources, but for the most part they are constructed inductively by dimension, and thus are not covariant under permutations of coordinates.

We will take as our starting point the operator defined in \cite{NijenhuisWoolf}, which we will call $Q$. Then we will modify it slightly to define an operator $P$ that satisfies our special needs. $Q$ has the following form: for $\theta$ a $(0,q)$-form on $B_1 \subset \C^n$:
\begin{align}\label{Q formula}
Q\theta = Q_1\pi_1\theta + Q_2\pi_2\theta + \ldots + Q_n\pi_n\theta
\end{align}
where the $Q_i$'s are certain integral linear operators acting independently on each of the coefficient functions, and the $\pi_j$'s are certain operators which ``remove'' factors of $d\bar{z}_j$ from the wedge product, as follows: if $\theta = \sum\limits_{i_1 < \ldots < i_q} a_{i_1\ldots i_q} d\bar{z}_{i_1} \^ \cdots \^ d\bar{z}_{i_q}$, then
\begin{align}
\pi_j \theta = \sum_{j < i_2 < \ldots < i_q} a_{j\, i_2 \ldots i_q} d\bar{z}_{i_q} \^ \cdots \^ d\bar{z}_{i_1}.
\end{align}
In words, $\pi_j\theta$ takes all components in $\theta$ whose lowest-index $d\bar{z}_i$ is in fact $d\bar{z}_j$, and just drops this $d\bar{z}_j$ out of the wedge product. Other components, whose lowest-index factor is not $d\bar{z}_j$, $\pi_j$ sends to zero.

Since $Q$ depends on the ordering of coordinates, we assumed (without loss of generality) that $S = \C^k \subset \C^n$ corresponds to the \emph{first} $k$ coordinates, with $z_{k+1} = \ldots = z_n = 0$ on $S$.  Then
\begin{lem}\label{Q isotropic}
If $\theta$ is a $(0,q)$-form for $q \geq 2$ with respect to which $S$ is istropic, i.e., for which $\iota^*(\theta)=0$, then $S$ is isotropic for $Q\theta$.
\end{lem}
In fact, Lemma \ref{Q isotropic} will also hold for the $P$ we define, as we shall see.
\begin{proof}
Since $S$ is isotropic for $\theta$, for any nonzero component---say, $a_{i_1\ldots i_q}\, d\bar{z}_{i_1}\^\ldots\^d\bar{z}_{i_q}$---there must be at least one factor---say, $d\bar{z}_m$---which is conormal to $S$, i.e., for which $m > k$. We consider terms $Q_j\pi_j\theta$ of $Q\theta$ falling into two cases:

Case 1: $j \leq k$ (the dimension of $S$). $\pi_j$ takes those components---say $a_{j\,i_2\ldots i_q}\, d\bar{z}_j  \^ d\bar{z}_{i_2}\^\ldots\^d\bar{z}_{i_q}$---which have a leading $d\bar{z}_j$ factor, and drops that factor.  In this case, since $j \leq k < m$, the conormal $d\bar{z}_m$ mentioned above was not $d\bar{z}_j$, and so it is still present among the remaining factors, $a_{j\,i_2\ldots i_q}\,d\bar{z}_{i_2}\^\ldots\^d\bar{z}_{i_q}$, and thus $S$ is still isotropic for these factors.  (Then $Q_j$ acts linearly on the coefficient functions without changing this fact.) 

Case 2: $j > k$. $\pi_j$ acts only on components---say $a_{j\,i_2\ldots i_q}\, d\bar{z}_{i_1}\^\ldots\^d\bar{z}_{i_q}$---of $\theta$ for which all the $j$ and $i_\bullet$'s are greater than $k$, in which case $S$ is once again isotropic for the remaining factors.
\end{proof}

$Q$ may also be defined on $\Omega^{0,q}\left(\^ ^{p,0} TB_r\right)$.  If $\alpha$ is a multi-index of length $p$ and $\theta_\alpha$ is a $(0,q)$-form, then let
\begin{align}\label{mixed tensor decomposition}
Q\left(\theta_\alpha \tens \frac{\del}{\del z^\alpha}\right) = Q(\theta_\alpha) \tens \frac{\del}{\del z^\alpha}
\end{align}
with extension to $\Omega^{0,q}\left(\^ ^{p,0} TB_r\right)$ by linearity.

Unfortunately, if we set $P=Q$ in general, the vector field part of $V(\epsilon)$ defined in \eqref{generalized homotopy operator} will not be $S$-preserving, given our assumptions on $\epsilon$.  We will define $P$ in a way that ``stabilizes $Q$ around $S$.''

Let $\iota : S \into \C^n$ be the inclusion and let $\rho : \C^n \to S$ be the projection.  For a $(0,q)$-form $\theta$, let
\begin{align}
s(\theta) = \rho^*(\iota^*(\theta))
\end{align}
In words, $s$ pulls back $\theta$ to $S$ and then stretches this form out over $\C^n$ again.  Then let
\begin{align}
P\theta = Q\theta - s(Q\theta) + Q\,s(\theta),
\end{align}
with extension to mixed tensors in $\Omega^{0,q}\left(\^ ^{p,0} TB_r\right)$ analogously to 
\eqref{mixed tensor decomposition}.  We note that $\bar\del$ commutes with $s$, and it follows that since $Q\bar\del + \bar\del Q = \Id$, also $P\bar\del + \bar\del P = \Id$, satisfying relation \eqref{homotopy relation}.

\begin{lem}\label{P isotropic}
If $\theta$ is an $S$-isotropic $(0,q)$-form for $q \geq 2$, then $P\theta = Q\theta$.  Thus, Lemma \ref{Q isotropic} holds for $P$ in place of $Q$.
\end{lem}
\begin{proof}
Since $S$ is isotropic for $\theta$, $\iota^*(\theta)=0$, and so $s(\theta)=0$.  Furthermore, from Lemma \ref{Q isotropic} we see that $s(Q\theta)=0$.  Thus,
\begin{align}
P\theta = Q\theta - s(Q\theta) + Q\,s(\theta) = Q\theta.
\end{align}
\end{proof}

\begin{lem}\label{P tangent}
If $\theta \in \Omega^{0,1}(\^ ^{p,0}TB_r)$ such that $\theta : T_{0,1}B_r \to \^ ^{p,0}TB_r$ takes $T_{0,1}S$ to $\^ ^{p,0}TS$, then $P\theta|_S \in C^\infty(\^ ^{p,0}TS)$.
\end{lem}
\begin{proof}
Consider a component of the form $f\, d\bar{z}_i \tens \frac{\del}{\del z^\alpha}$, where $f$ is a function and $\alpha$ is a multi-index.  If $i \leq k$, then by the hypothesis $\frac{\del}{\del z^\alpha}$ is multi-tangent to $S$, and so, regardless what the coefficient $P(fd\bar{z}_i)$ turns out to be, the result holds for this component.

On the other hand, if $i > k$, $\frac{\del}{\del z^\alpha}$ is \emph{not} tangent to $S$, so we have to hope that the coefficient function $P(fd\bar{z}_i)$ vanishes on $S$.  But $d\bar{z}_i$ is conormal to $S$, so
\begin{align}
P(fd\bar{z}_i) &= Q(fd\bar{z}_i) - s(Q(fd\bar{z}_i)) + Q\,s(fd\bar{z}_i) \notag \\
&= Q(fd\bar{z}_i) - s(Q(fd\bar{z}_i)) + 0.
\end{align}
On $S$, the functions $Q(fd\bar{z}_i)$ and $s(Q(fd\bar{z}_i))$ are equal, so $P(fd\bar{z}_i)|_S = 0$, as desired.
\end{proof}

We have the ingredients we need to prove that our generalized vector field is ``tangent'' to $S$:
\begin{prop}\label{V tangent}
In the setup we have defined, with $S = \C^k \subset \C^n$ a brane with generalized tangent bundle $\tau = TS\dsum T^*S$, $\epsilon = \epsilon_{2,0} + \epsilon_{1,1} + \epsilon_{0,2}$ with $\epsilon_{2,0} \in C^\infty(\^ ^{2,0} TB_r)$ making $S$ coisotropic, $\epsilon_{1,1} \in C^\infty(T_{1,0}B_r \tens T^*_{0,1}B_r)$ preserving $TS$, and $\epsilon_{0,2} \in \Omega^{0,2}(B_r)$ making $S$ isotropic, and $P$ as defined, the generalized vector field
\begin{align}\label{V expressed again}
V(\epsilon) := P([\epsilon_{2,0},P\epsilon_{0,2}] - \epsilon_{1,1} - \epsilon_{0,2}),
\end{align}
lies in $\tau$ when restricted to $S$.
\end{prop}
\begin{proof}
Starting with the rightmost term of \eqref{V expressed again}, $P\epsilon_{0,2}$ lies in $T^*S$ by Lemma \ref{P isotropic}.  $P\epsilon_{1,1}$ lies in $TS$ by Lemma \ref{P tangent}.  Finally, since $P\epsilon_{0,2}|_S \in C^\infty(N^*_{0,1}S)$, therefore $[\epsilon_{2,0},P\epsilon_{0,2}]|_S \in C^\infty(T_{1,0}B_r|_S \tens N^*_{0,1}S)$; thus, $[\epsilon_{2,0},P\epsilon_{0,2}]|_S : T_{0,1}S \to 0$ and, by Lemma \ref{P tangent}, $P[\epsilon_{2,0},P\epsilon_{0,2}]|_S$ lies in $TS$.
\end{proof}

\begin{cor}
Under the above hypotheses, the generalized flow generated by $V(\epsilon)$ leaves $S$ and $\tau$ invariant.
\end{cor}
\begin{proof}
By Proposition \ref{V tangent}, the real vector part, $X$, of $V(\epsilon)$ is tangent to $S$, so the diffeomorphism, $\phi_t$, of the generalized flow leaves $S$ invariant.  If $\xi$ is the real 1-form part of $V(\epsilon)$, then, as per \eqref{flow B field}, the $B$-transform part of the generalized flow is $d\xi_t$, where
\begin{align}
\xi_t := \int_{s=0}^t \phi_s^*(\xi) ds.
\end{align}
Of course, $\xi_t$ is conormal to $S$ at all $t$.  For $Y \in TS$ and $\eta \in N^*S$, $d\xi_t$ acts on some $Y+\eta \in \tau$ via
\begin{align}
e^{d\xi_t}\cdot (Y+\eta) &= Y + \iota_Y d\xi_t + \eta
\end{align}
Since $\xi_t \in N^*S$ and $\iota_Y\xi_t = 0$, it follows that $\iota_Y d\xi_t \in N^*S \subset \tau$, as required.
\end{proof}

\section{The algorithm}\label{the algorithm}

We briefly review the algorithm in \cite{bailey2013} that proves Theorem \ref{gc local}, so that we can see how our modification along with the ansatz proves our results (Main Lemma \ref{main lemma} and Theorem \ref{brane local form}).  Since the modification is minor, we do not repeat the details here---they may be found in \cite{bailey2013}. The algorithm is inspired by Conn \cite{Conn1985}, whose basic idea is formalized in a technical result of Miranda-Monnier-Zung \cite{MMZ}.

Let $\epsilon = \epsilon_{2,0} + \epsilon_{1,1} + \epsilon_{0,2}$ be a generalized complex deformation of the standard complex structure on a polydisc, $B_r \subset \C^n$, of radius $r$, where $\epsilon_{2,0}$ a $(2,0)$-bivector, $\epsilon_{1,1}$ a section of $T_{1,0}B_r \tens T^*_{0,1}B_r$, and $\epsilon_{0,2}$ a $(0,2)$-form, and where $\epsilon$ is Maurer-Cartan \eqref{MC}.

With the bracket and the $\bar\del$-homotopy operator $P$, we construct the generalized vector field $V(\epsilon) = P([\epsilon_{2,0},P\epsilon_{0,2}] - \epsilon_{1,1} - \epsilon_{0,2})$.  The infinitesimal action of $V(\epsilon)$ on the deformation $\epsilon$ gives a new deformation $V\cdot\epsilon = \bar\del V + [V,\epsilon]$ which, by design, has $(1,1)$ and $(0,2)$ parts which are quadratically small in $\epsilon_{1,1}$ and $\epsilon_{0,2}$ (see \cite[Lemma 6.9]{bailey2013}.

Similar quadratic estimates hold for the time-1 generalized flow, $\phi_V \cdot \epsilon$, of $V(\epsilon)$ acting on $\epsilon$ (see \cite[Lemma 6.11]{bailey2013}); therefore, if we were to iterate the process, with $\epsilon^{n+1} = \phi_{V(\epsilon^n)} \cdot \epsilon^n$, we would hope that the limit $\epsilon^\infty := \lim_{n\to\infty}\epsilon^n$ had only a bivector component, and \eqref{MC} would tell us that it was holomorphic and Poisson.

There are two issues.  First, $\phi_V$ is only a local diffeomorphism about $0 \in B_r$, and the image of $B_r$ under $\phi_V$ does not contain $B_r$, but rather a smaller ball $B_{r'},\, r'<r$.  This is not much of a problem: if we take care that the radius does not shrink too fast, then in the limit $\epsilon^\infty$ is still defined on a positive radius.

Second, at each stage of the iteration we ``lose derivatives,'' so that our norm estimates for the $k$-th derivative of $\epsilon^{n+1}$ will depend on various $k+l$-th derivatives of $\epsilon^n$ for some $l>0$, and in the limit we expect no convergence at all. The solution, due to Nash and Moser (see \cite{Hamilton} for a review of the general approach), is to apply ``smoothing operators,'' $S_t$, to $V(\epsilon^n)$ at each stage, which gain back some derivatives, at the cost of making an approximation.  If the ``strength,'' $t$, of the smoothing at each stage is wisely chosen to balance these two concerns, then indeed we will get convergence to an appropriate $\epsilon^\infty$.

Of course, this only works if $\epsilon$ is a small deformation in some sense. A bit of work (see \cite[Section 7]{bailey2013}) goes into showing that, on a suitably small neighbourhood of the points in question, a generalized complex structure meets this criterion, but in the end Theorem \ref{gc local} follows.

\subsection{The modification, and the Main Lemma}\label{main lemma section}

In Definition \ref{ansatz def} we have specified the ansatz that the deformed generalized complex structure admits a rank-0 brane $S = \C^k \subset \C^n$ with generalized tangent bundle $\tau = TS \dsum N^*S$. Whereas in \cite{bailey2013}, $P$ only needed to be a homotopy operator for $\bar\del$ satisfying some estimates, here we have chosen a special $P$ such that $V(\epsilon)|_S$ lies in $\tau$, and thus the generalized flow of $V(\epsilon)$ preserves $S$ and $\tau$.  Our $P$ satisfies all of the same formal properties as in \cite{bailey2013}, and so our choice does not affect the proof.  A point of concern is that the smoothed $S_t(V(\epsilon))$ may no longer satisfy the generalized tangency condition.  However, it is straightforward to construct smoothing operators which act only the space of generalized vector fields tangent to $S$, in much the same way as one would otherwise construct smoothing operators (once again, see \cite{Hamilton} for details).

Since our homotopy and smoothing operators specialize the requirements of \cite{bailey2013}, the whole proof of the Main Lemma of \cite{bailey2013} goes through without change; and since the algorithm preserves the ansatz, we have
\begin{thm:mainlem}\label{main lemma}
In the set-up and notation of Section \ref{the algorithm}, if $\epsilon$ is small enough, there is a local generalized diffeomorphism $\phi$ in a neighbourhood of $0 \in \C^n$, preserving $S=\C^k$ and $\tau = TS \dsum N^*S$, such that $\phi\cdot\epsilon$ is a holomorphic (on $\phi(B_r) \subset \C^n$) Poisson bivector.
\end{thm:mainlem}

\section{The linear case}\label{linear case}

In order to justify the ansatz of Definition \ref{ansatz def}, we must first establish a linear version of Theorem \ref{brane local form}.

Let $V$ be a real vector space, and let $\mathbb{V} \iso V \dsum V^*$ be a ``linear exact Courant algebroid,'' that is, $\mathbb{V}$ has a nondegenerate symmetric pairing, and an anchor map $a : \mathbb{V} \to V$ such that
\begin{align}
0 \to V^* \oto{a^*} \mathbb{V} \oto{a} V \to 0 \label{linear exact sequence}
\end{align}
is exact.  (As usual, we treat $V^* \subset \mathbb{V}$ as a subspace.)  Let $\II : \mathbb{V} \to \mathbb{V}$ be a linear generalized complex structure on $V$---i.e., an orthogonal anti-involution of $\mathbb{V}$---with real bivector $P := a \comp \II |_{V^*} : V^* \to V$, and suppose we have the data of a ``linear rank-0 brane'', i.e., a ``supporting subspace'' $S \subset V$ along with a ``generalized tangent bundle'' $\tau \subset \mathbb{V},$ such that $\tau$ is an $\II$--invariant maximal isotropic and $a(\tau) = S$.
\begin{prop}\label{linear result}
If $\dim_\R S = 0 \mod 2$, then there exists a splitting $s : V \to \mathbb{V}$ of the sequence \eqref{linear exact sequence} and a complex structure $I : V \to V$ and a $B$-field $B \in V^* \^ V^*$ such that
\begin{itemize}
\item $\II$ is ``linear complex Poisson'' with respect to $\mathbb{V} = s(V) \dsum V^*$, i.e.,\begin{align}\label{linear hol form}
\II =
\begin{pmatrix}
-I & P\\
0 & I^*
\end{pmatrix}
\begin{array}{l}
\scriptstyle{s(V)} \\
\scriptstyle{\;V^*}
\end{array}
\end{align}
for some complex structure $I : V \to V$,
\item $S \subset V$ is a complex subspace for $I$, and
\item $\tau = s(S) \dsum N^*S$.
\end{itemize}
\end{prop}
\begin{proof}
We will choose an isotropic, $\II$-invariant $U := s(V) \subset \mathbb{V}$ which is compatible with $\tau$ in a certain way.  We will specify $U$ as the direct sum of three parts: $U_N$, where $a(U_N)$ is complementary to $S$, $U_P$, where $a(U_P) = P(N^*S) \subset S$ (recall that branes are isotropic), and $U_S$, where $a(U_S)$ is complementary to $a(U_P)$ in $S$.

Let $W =N^*S + \II\left(N^*S\right)$---an $\II$-invariant subspace of $\tau$.  We let $U_S$ be any $\II$-invariant complement to $W$ in $\tau$ (which exists, by the existence of $\C$-linear vector space complements).

Let $U_P$ be an $\II$-invariant complement to $N^*S$ in $W$.  Even though $N^*S$ is not itself $\II$-invariant, we can see that such a complement exists: let $e_1,\ldots,e_q$ be a basis for ${N^*S \,\cap\, \II\left(N^*S\right)}$ (a complex, and thus $\R$-even-dimensional space).  To get a basis for $W$, we need additional basis elements complementary to ${N^*S \,\cap\, \II\left(N^*S\right)}$.  From the dimension hypothesis we find there are an even number, thus we pair them, as $(f_i,g_i)$, etc.  The $\II f_i$'s and $\II g_i$'s are complementary to $N^*S$ in $W$, and an $\II$-invariant complement is generated by $f_i + \II g_i$, $\II f_i - g_i$, etc.

Since $U_P$ and $U_S$ are contained in $\tau$, $U_P \dsum U_S$ is isotropic; furthermore, since it is complementary to $N^*S = V^* \cap \tau$ in $\tau$, it covers $S$.

Now consider the $\II$-invariant space $\left(U_P \dsum U_S\right)^\perp \subset \mathbb{V}$.
\begin{align}
a\left(\left(U_P \dsum U_S\right)^\perp\right) &= \Ann\left(V^* \cap \left(U_P \dsum U_S\right)\right) \subset V \notag \\
&= \Ann(0 \subset V^*) = V.
\end{align}
$\left(U_P \dsum U_S\right)^\perp$ contains $\tau$, and thus there exist $\II$-invariant complements to $\tau$ in $\left(U_P \dsum U_S\right)^\perp$.  Furthermore, since $\II$ is orthogonal, we may choose such a complement, $U_N$, to be isotropic.  Then,
\begin{align}
V = a\left(\left(U_P \dsum U_S\right)^\perp\right) &= a(U_N \dsum \tau) \notag \\
&= a(U_N \dsum U_P \dsum U_S \dsum N^*S) \notag \\
&= a(U_N \dsum U_P \dsum U_S),
\end{align}
so $U := U_N \dsum U_P \dsum U_S$ is a maximal isotropic, $\II$-invariant subspace covering $V$. 

Because $U$ is $\II$-invariant, the splitting $s : V \isoto U$ gives us $\II$ in the form \eqref{linear hol form}.  Thus it gives us a complex structure, $I = \II|_U$, on $V$, for which $S \iso U_P \dsum U_S$ is a complex subspace.  Furthermore, $\tau = U_P \dsum U_S \dsum N^*S = s(S) \dsum N^*S$.
\end{proof}

\begin{rem}
The parity condition that appears here is different from that in Theorem \ref{gc local} from \cite{bailey2013}. In order to get a local holomorphic Poisson structure there, we required that the rank of the real Poisson structure was $0 \mod 4$. Here, however, we only require that the brane be even-dimensional---the existence of a brane constrains the generalized complex structure to be ``more like'' a holomorphic structure.
\end{rem}

\subsection{The local form theorem}\label{final proof section}

We may now give the final proof of Theorem \ref{brane local form}. Recall:

\begin{thm:main}
If $S$ is an $\R$-even-dimensional, rank-0 generalized complex brane in $(M,\II)$ with generalized tangent bundle $\tau \subset \Tc M|_S$, then any $p \in S$ has a neighbourhood in $M$ on which the data $\II$, $S$ and $\tau$ are equivalent, under a $B$-transform, to a holomorphic Poisson structure $\pi$ for some complex structure $I$, with brane $S$ equal to a complex, coisotropic submanifold with generalized tangent bundle $TS \dsum N^*S$.
\end{thm:main}
\begin{proof}[Proof, part 1: satisfying the ansatz]
\renewcommand{\qedsymbol}{}
First we show that $p$ has a neighbourhood in which the data may be put in the form of the ansatz  of Definition \ref{ansatz def}.
	
We invoke Proposition \ref{linear result} at $p$ to get a complex structure $I_p$ on $T_p M$ and a splitting of $\Tc_p M$ such that $\II_p$ is an $I_p$-complex Poisson structure as in \eqref{linear hol form}.  In particular, $\II_p$ is a deformation of $I_p$ by some bivector $\epsilon_p$.  Since $T_pS$ is $I_p$--invariant, we may choose complex coordinates near $p$ such that $S$ is a complex submanifold. Let $W \iso B_r \subset \C^n$ be a neighbourhood of $p$ which is a polydisc in these coordinates. As per Remark \ref{generalized tangent is locally standard}, we may extend the splitting $\Tc_p M$ along $W \cap S$, compatibly with $\tau$, such that, in the splitting, $\tau \iso TS \dsum N^*S$; then the splitting on $W \cap S$ is extended to $W$ in an arbitrary way.  $\II|_W$ is a deformation of $I$ by some $\epsilon$ extending $\epsilon_p$.  From $I$ we get an isomorphism between $W$ and a polydisc $B_r \subset \C^n$, with $S \cap W \iso \C^k \cap B_r$, and from the splitting we get an isomorphism $\Tc W \iso \Tc B_r$.

This puts us in the setting of the ansatz.  It follows from Main Lemma \ref{main lemma} that, if $\epsilon$ is small enough in a certain sense, then there is a generalized diffeomorphism, $\phi$, from a neighbourhood of $p$ to a neighbourhood of $0 \in \C^n$ such that $\phi_*(\II)$ is holomorphic Poisson, $S$ goes to $\C^k$, and $\tau$ goes to $T\phi(S) \dsum N^*\phi(S)$.  If we pull the complex structure back through $\phi$, then we may drop the diffeomorphism part of $\phi$, and we have a $B$-transform on a neighbourhood of $p$ putting $(\II,S,\tau)$ in the desired ``holomorphic coisotropic form.''
\end{proof}

\begin{proof}[Proof, part 2: smallness]
\renewcommand{\qedsymbol}{}
Is $\epsilon$ small enough for the Main Lemma \ref{main lemma}?  The answer is ``yes, after some restriction and rescaling.'' This is addressed in detail in \cite[Section 7]{bailey2013} (where one should look for details), but under slightly different hypotheses, so we give a brief explanation here.  We have two operations we may use to make $\epsilon$ smaller:

First, we may ``zoom in,'' i.e., dilate a small neighbourhood of $p$ to a ball of radius $1$ in $\C^n$, and discard the rest.  This allows us to take advantage of any vanishing $\epsilon$ may exhibit at $p$.  However, the different components of $\epsilon$, having different co- and contravariant degrees, have different scaling laws.  If $\epsilon_{2,0}$, $\epsilon_{1,1}$ and $\epsilon_{0,2}$ vanish to order $\alpha$, $\beta$ and $\gamma$ respectively at $p$, then zooming by a factor $t > 0$ scales $\|\epsilon_{2,0}\|$, $\|\epsilon_{1,1}\|$ and $\|\epsilon_{0,2}\|$ by $t^{\alpha - 2}$, $t^\beta$ and $t^{\gamma+2}$ respectively (up to a constant).

The second operation is to ``scale'' $\epsilon$ by scaling $T^*M$ in $\Tc M$.  This is not a generalized diffeomorphism, but does take generalized complex structures to generalized complex structures.  For $s > 0$, we define
\begin{align}
\lambda_s : \epsilon_{2,0} + \epsilon_{1,1} + \epsilon_{0,2} \mapsto s\,\epsilon_{2,0} + \epsilon_{1,1} + s^{-1}\epsilon_{0,2}.
\end{align}

In the ansatz, $\epsilon_p$ is complex Poisson for $I_p$.  In other words, $\epsilon_{2,0}$ may not vanish at $p$, but $\epsilon_{1,1}$ and $\epsilon_{0,2}$ vanish to order $1$.  Then if we zoom by factor $t$ and apply the scaling $\lambda_s$ for $s=t^{\frac{5}{2}}$, then $\|\epsilon_{2,0}\|$, $\|\epsilon_{1,1}\|$ and $\|\epsilon_{0,2}\|$ scale by $t^{\frac{1}{2}}$, $t$ and $t^{\frac{1}{2}}$ respectively (up to a constant); thus, for small enough $t$, we have a zoomed, scaled $\epsilon$ as small as we like.
\end{proof}

\begin{proof}[Proof, conclusion]
Taking Part 1 and Part 2 together, $(\II,S,\tau)$ has a holomorphic standard form near $p$, and the fact that $S$ is coisotropic with respect to the holomorphic Poisson structure $\pi = -\tfrac{1}{4}(iP + IP)$ is a straightforward translation of what it means to be a generalized complex brane in the holomorphic Poisson context.
\end{proof}

\subsection{Higher-rank branes} We can augment Theorem \ref{brane local form} to describe situations in which the brane supports a nontrivial generalized holomorphic vector bundle.

\begin{cor}\label{higher rank local form}
Suppose we have a generalized complex brane in $(M,\II)$ determined by the following data: the support $S \subset $ with $\dim_\R S$ even-dimensional, generalized tangent bundle $\ell = \tau \subset \Tc M|_S$ such that $\II(\tau) = \tau$, and complex vector bundle $V \to S$ with $\tau^{1,0}$--connection $\nabla$.

Then any $p \in S$ has a neighbourhood in $M$ on which the brane data are equivalent, up to a $B$-transform, to Example \ref{hol coiso example}, that is, to the data $(S,\tilde{\tau},\tilde{I},\tilde{\nabla})$, where
\begin{itemize}
	\item $\tilde{\II}$ is of holomorphic Poisson for some complex structure $I$,
	\item $\tilde{\tau} = TS \dsum N^*S$,
	\item $\tilde{\nabla}_{X + (1+\pi)\xi}\, v = \bar\del_X v + \nabla^\pi_\xi v$ for $X \in T_{0,1} S$ and $\xi \in N^*_{1,0} S$, where $\nabla^\pi_\cdot$ is a holomorphic Poisson connection on $S$ \cite{fernandes2000}.
\end{itemize}
\end{cor}

\begin{proof}
By dropping $V$, we have a rank-0 brane, and we are in the case of Theorem \ref{brane local form}.  Thus, without loss of generality, we assume that $\II$ is indeed holomorphic Poisson, $S$ is a complex coisotropic submanifold, and $\tau = TS \dsum N^*S$ with $\ell = T_{0,1}S \dsum \Gamma_\pi$, where $\Gamma_\pi := (1+\pi)(N^*_{1,0}S) \iso N^*_{1,0}S$.

We examine the action of $\ell$.  The following argument is essentially in \cite[Proposition 8]{Gualtieri2010}.  Decompose $\nabla$ as $\nabla' + \nabla''$, where $\nabla' : C^\infty(T_{0,1}S \tens V) \to C^\infty(V)$ and $\nabla'' : C^\infty(\Gamma_\pi \tens V) \to C^\infty(V)$.  From the condition $\nabla^2=0$, we see that $\nabla'$ squares to zero and is just an action of $T_{0,1}S$ on $V$, i.e., a holomorphic structure.  Furthermore, $\nabla''$ anticommutes with $\nabla'$, i.e., $\nabla''$ is holomorphic.  Finally, $\nabla''$ is a flat connection for the Lie algebroid $\Gamma_\pi$, which is just a Poisson connection (see \cite{fernandes2000}).
\end{proof}

\section{Example: the Hopf surface}\label{example section}

The standard Hopf surface is a complex surface defined as
\begin{align}
X := (\C^2 \setminus \{0\}) \,/\, \bb{Z},
\end{align}
where $1 \in \bb{Z}$ acts by multiplication by 2. Of course, the standard complex structure on $X$ determines a generalized complex structure. However, in \cite{Gualtieri2014}, and elaborated in \cite{BG2016}, $X$ admits some interesting generalized complex structures (in fact, generalized Kahler structures) which are not deformations of the underlying complex structure---nor do they admit any global holomorphic gauge. We will define one of these structures concretely (omitting its derivation from bi-Hermitian geometry, which is in \cite{BG2016}), and then we will define some branes on it whose existence is not obvious without the ``locally holomorphic'' view.

 Let $(z_1,z_2)$ be the standard coordinates on $\Z^2$, and let $R^2 = z_1 \bar{z}_1 + z_2 \bar{z}_2$. We define a complex, nondegenerate 2-form on $\C^2 \setminus \{0\}$,
 \begin{align}\label{define C}
 C = \frac{1}{R^2} \left( \frac{2x_1}{\bar{x}_2} d\bar{x}_1 \^ d\bar{x}_2 \,+\, dx_1 \^ d\bar{x}_1 \,+\, dx_2 \^ d\bar{x}_2 \right).
 \end{align}
$C$ is defined everywhere except on the special locus
\begin{align*}
D := \{(x_1,x_2) | x_2 \neq 0\},
\end{align*}
which we call the \emph{complex locus}.

Since $C$ is homogenous of degree $0$, it passes through the $\Z$-action to $X$ (away from $\{x_2=0\}$, which we continue to denote as $D \subset X$, hopefully without ambiguity). $C$ transforms $TX \subset \Tc X$ to a complex Dirac structure
\begin{align}
L := e^C\, TX = (1 + C) TX \;\subset\; TX \dsum T^*X,
\end{align}
which extends smoothly to $D$. $C$ is not closed, but
\begin{align}\label{dC}
dC = \frac{1}{R^4}\Big[ (x_2 d\bar{x}_2 - \bar{x}_2 dx_2) dx_1 \^ d\bar{x}_1 \,+\, 
(x_1 d\bar{x}_1 - \bar{x}_1 dx_1) dx_2 \^ d\bar{x}_2 \Big],
\end{align}
is real, so $L$ is an integrable Dirac structure in $T_\C M \dsum T_\C ^*M$ with bracket twisted by $H = -dC$. Indeed, $L$ is the $+i$-eigenbundle of a generalized complex structure, $\II$, on $X$.

$\II$ is of symplectic type (up to $B$-transform) everywhere except $D$, where it is of complex type. Some easy-to-see branes on $X$ are the whole of $X$, $D$, points in the complex locus, and Lagrangians not intersecting $D$. Here we give some other examples:

\subsection{A family of $S^2$ branes on $X$}
Let $c$ be any positive real number, and let
\begin{align}\label{brane equation}
S = \{ (x_1,x_2) \,|\, x_1 \,\textnormal{is real and}\, R^2 = c\}.
\end{align}
Strictly speaking, we allow $R^2 = 4^k c$ for all $k \in \Z$, since this constraint is invariant under the $\Z$-action. $S$ is diffeomorphic to the 2-sphere in the obvious way.

\begin{prop}
There is a unique rank-0 brane supported on $S$.
\end{prop}

\begin{proof}
We would like to realize $S$ as a holomorphic coistropic submanifold; however, as we said, $\II$ does not admit a global holomorphic gauge. But we can define new \emph{local} complex coordinates
\begin{align}
w_1 = \log\left(\bar{x}_1 R^2 / x_1\right)  \qquad w_2 = \bar{x}_2 / R.
\end{align}
$w_2$ is defined on all of $\C^2 \setminus \{0\}$ and $w_1$ is defined away from the locus $Z := \{(x_1,x_2) | x_1 \neq 0\}$. The $\Z$-action is holomorphic in these coordinates, so they define a complex structure on $\tilde{X} := X \setminus Z$.

We define a $B$-transform on $\tilde{X}$ via the real 2-form
\begin{align}
B = \frac{x_2 \bar{x}_2}{2R^2} d\log\left(\frac{\bar{x}_1}{x_1}\right) d\log\left(\frac{\bar{x}_2}{x_2}\right)
\end{align}
and let $\JJ = e^B \II e^{-B}$. Since $dB = H = -dC$, $\JJ$ is defined on the \emph{untwisted} standard Courant algebroid $T\tilde{X} \dsum T^*\tilde{X}$.

$\JJ$, like $\II$, is given by a complex Dirac structure of the form $e^W\; T\tilde{X}$, where $W = C + B$. A caculation shows that
\begin{align}
W = d w_1 \^ d \log w_2.
\end{align}
This is a holomorpic \emph{log-symplectic structure} \cite{GUILLEMIN2014864}. It inverse to the holomorphic Poisson structure,
\begin{align}
\pi = w_2\, \del_{w_1} \^ \del_{w_2},
\end{align}
and the generalized complex structure $\JJ$ on $\tilde{X}$ is determined . (Details on how these formulas were obtained can be found in \cite[Section 8.4]{BG2016}.)

In these coordinates, the defining equation \eqref{brane equation} of $S$ is just $w_1 = \log c$. Thus, $S$ is a holomorphic coistropic in $\tilde{X}$ intersecting the complex locus $D$. (It is coisotropic because, in these complex coordinates, it is just a curve in a surface.) We define the generalized tangent bundle in this gauge as $\tau_S = TS \dsum N^*S$, as in Example \ref{hol coiso example}. But away from $D$, where $X$ is of symplectic type, $\tau_S$ was already determined completely by the fact that $S$ is Lagrangian, and so $\tau_S = \JJ(N^*S) \dsum N^*S$, as in Example \ref{lagrangian example}. Clearly, these $\tau_S$'s must coincide, but the latter definition extends to $Z$, so we have a generalized tangent bundle for all of $S \subset X$.
\end{proof}

Regarding higher-rank branes: if $\{p\} = S \cap D$, then on $S \setminus \{p\}$, as per Example \ref{lagrangian example}, higher rank branes are just flat vector bundles. \emph{A priori}, we might expect some choice in how the connection extends to $p$, i.e., we would have a \emph{meromorphic connection} with some monodromy around $p$. However, $S \setminus \{p\}$ is simply connected, so higher rank branes will have no monodromy, and will just be trivial vector bundles over $S$.

\section{Connection to quantization}

The A-model of a symplectic (or, hypothetically, generalized complex) manifold, is a kind of ``quantization'' of the manifold. As we said, generalized complex branes are the natural notion of ``brane'' in the (Kapusting-Orlov) A-model. This physics model is presumed to correspond mathematically to the ``Fukaya category'' of the manifold, with branes as (generating) objects and pseudoholomorphic curves stretched between them as (generating) morphisms. A major goal of generalized complex geometry is to define the ``Fukaya category'' of a generalized complex manifold. Presumably, the objects of this category will be generated by generalized complex branes. Thus, we hope this current paper on their local structure, and upcoming work giving the corresponding global perspective, are steps towards understanding this category.

But there is another quantization context in which branes---and the results of this paper---hove more potential to be immediately useful. Branes also appear in \emph{deformation quantization}. Given a Poisson structure on a smooth manifold, Kontsevich constructs a $\star$-product (formal in parameter $\hbar$) \cite{Kontsevich1997} deforming the algebra of functions, for which the Poisson structure is the commutator to first order. He extends this construction to the setting of complex geometry \cite{Kontsevich2001}, where the global quantization data must now be a more general, higher-categorical object than a sheaf of algebras (he calls it a ``stack of algebroids''). However, locally, the quantization goes through just as in the smooth case, and so about any point one has (up to ``inner automorphism'') a $\star$-product on the holomorphic functions.

Though deformation quantization of a generalized complex manifold is still an ongoing project, there is strong evidence that the deformation quantization should be the same sort of thing as the quantization of a holomorphic Poisson manifold \cite{Kontsevich2001}, generalized somewhat to a \emph{Hopf algebroid}. (The ``strong evidence'' is that the deformation theory of a ``stack of algebroids'' is the same as the generalized complex deformation theory; indeed, this is one of the principal ways in which generalized complex geometry naturally arises.) As we said in Section \ref{upcoming paper section} above, the generalized complex manifold should be viewed as a holomorphic \emph{stack}, and therefore in some sense we should be quantizing this stack. ``Derived higher stacks'' are deformation-quantized in \cite{CPTVV2015} in the very general setting of derived algebraic geometry; but in an ongoing project with Gualtieri, we work to make the quantization concrete and explicit as it applies to generalized complex geometry, and to make it more directly comparable to Kontsevich's ``stacks of algebroids'' construction in \cite{Kontsevich2001}.

Having sketched the picture of deformation quantization, we comment on how branes fit into it. In the Poisson context, Poisson modules over coisotropic submanifolds are the semi-classical (i.e., first-order in $\hbar$) limits of $\star$--product bimodules. Because of cohomological obstructions (studied by Cattaneo and Felder \cite{CattaneoFelderRelativeFormality}), not every coisotropic survives deformation quantization to become a global bimodule in a strict sense, but in a weak sense coistropics should be the things that quantize to bimodules.

Then the results of this paper show that, given a suitable local choice of holomorphic gauge realizing a brane as a holomorphic coisotropic, we have a local deformation quantization of the holomorphic functions, for which the brane should (up to possible obstructions) quantize to a bimodule. Fitting this into the global picture awaits a full definition of generalized complex deformation quantization, but this paper (and the upcoming global version) is one of the steps toward establishing this.

Finally, tying these two notions of quantization together, Kapustin \cite{Kapustin2004} suggests (for physics reasons) that the deformation quantization of a holomorphic symplectic manifold (in the sense of \cite{Kontsevich2001} or \cite{NestTsygan}) will in fact turn out to be equivalent to the A-model of the underlying real symplectic structure, allowing one to bypass the difficult issue pseudoholomorphic curves. Parallel to this approach is the recent work of Ben-Bassat, Brav, Bussi and Joyce (in \cite{BBBBJ} and several earlier papers) to build a ``Fukaya category'' for holomorphic symplectic manifolds using the methods of shifted symplectic structures and perverse sheaves, and which once again bypasses the need to work with nontrivial holomorphic curves.

Though our perspective on generalized complex geometry as locally holomorphic or (globally) ``holomorphic on the stack'' is not directly analogous to the holomorphic symplectic case, the examples above suggest that, in coping with the harder aspects of the A-model, there may be something to be gained from treating things in a holomorphic way. This is a direction for future research.

\bibliographystyle{hyperamsplain-nodash}
\bibliography{references}{}

\end{document}